\documentclass[11pt]{amsart}
\pdfoutput=1
\usepackage[utf8]{inputenc}
\usepackage[T1]{fontenc}
\usepackage[a4paper,text={460pt,660pt},headsep=8mm, centering,marginparwidth=2cm]{geometry}
\usepackage[hidelinks,pdfusetitle]{hyperref}
\usepackage[activate={true,nocompatibility},final,tracking=true,kerning=true,factor=1100,stretch=10,shrink=10]{microtype}
\usepackage{braket}    %%%% \bracket
\usepackage{amsthm}    %%%% theorems
\usepackage{amssymb}   %%%% \precurly
\usepackage{mathrsfs}  %%%% \mathsf{L}
\usepackage[shortlabels]{enumitem}
\usepackage{caption}
\usepackage[margin=5pt,justification=centering,labelformat=simple,labelfont=sc]{subcaption}
\usepackage[mode=buildnew]{standalone} %%% for the figures
\usepackage{tikz-cd}
\usepackage{pgfplots}
\pgfplotsset{compat=newest}

\usepackage{bbold}
\theoremstyle{plain}
\newtheorem*{theo*}{Theorem}
\newtheorem{theorem}{Theorem}[section]

\newtheorem{Vtheo}{Theorem}

\newtheorem{Vprop}[Vtheo]{Proposition}
\newtheorem{Vlem}[Vtheo]{Lemma}

\newtheorem{Vdefi}[Vtheo]{Definition}

\theoremstyle{remark}

\newtheorem*{rem*}{Remark}

\newtheorem{Vrem}[Vtheo]{Remark}
\newtheorem{Vex}[Vtheo]{Example}

\DeclareMathOperator\sgn{sign}               % supremum
\DeclareMathOperator\spec{Spec}               % supremum
\newcommand{\ce}{\ensuremath{\mathscr{E}}}

\newcommand{\ind}[1]{\mathbb{1}_{#1}}
\newcommand{\un}{\mathbb{1}}

\newcommand{\ca}{\ensuremath{\mathscr{A}}}

\newcommand{\R}{\ensuremath{\mathbb{R}}}

\newcommand{\C}{\ensuremath{\mathbb{C}}}
\newcommand{\N}{\ensuremath{\mathbb{N}}}

\newcommand{\Z}{\ensuremath{\mathbb{Z}}}

\newcommand{\Diag}{\ensuremath{\mathrm{Diag}}}

\date{\today}
\author{Jean-François Delmas}
\address{Jean-François Delmas,
  CERMICS, \'{E}cole des Ponts, France}
\email{jean-francois.delmas@enpc.fr}

\author{Dylan Dronnier}
\address{Dylan Dronnier,
  Université de Neuchâtel, Switzerland}
\email{dylan.dronnier@unine.ch}

\author{Pierre-André Zitt}
\address{Pierre-André Zitt, LAMA, Université Gustave Eiffel, France}
\email{pierre-andre.zitt@univ-eiffel.fr}

\newcommand{\diag}{\ensuremath{\mathrm{Diag}}}
\DeclareMathOperator\sspec{spec}
\DeclareMathOperator\card{Card}

\title{Transformations preserving the effective spectral radius of a matrix}

\usepackage[draft]{fixme}
\FXRegisterAuthor{pa}{apa}{PAZ}
\FXRegisterAuthor{jf}{jf}{JFD}
\fxusetheme{colorsig}

\begin{document}

\thanks{This work is partially supported by Labex B\'ezout  reference ANR-10-LABX-58}

% 15A86: Linear preserver problems
% 15A18: Eigenvalues, singular values, and eigenvectors
% 15A15: Determinants, permanents, traces, other special matrix functions

\begin{abstract} 
 We discuss transformations on matrices that preserve the effective spectrum and/or the effective spectral radius.
\end{abstract}

\maketitle

\section{Effective spectral radius of a matrix}

Let $n\in\N^*=\{n>0\, \colon \, n\in \Z\}$.
 For   $K$  a square matrix of size  $n$, let $\sspec(K)$ and
 $\rho(K)=\max\{|\lambda|\, \colon \lambda \in \sspec(K)\}$ denote
 the spectrum and spectral radius of $K$. By the Perron-Frobenius theorem, if  $K$ has nonnegative entries, then
 $\rho(K)$    is    an    eigenvalue   and    thus    belongs    to
 $\spec(K)$.  For  $\eta\in\R_+^n$,
  let $\Diag (\eta)$ denote the diagonal matrix with entries given by
  $\eta$.
  %  and let  $K\eta$ denote the square 
  %   matrix $K\, \diag(\eta)$.
  %   , defined by:
  % \[
  %   (K\eta)_{ij} = K_{ij}\,  \eta_j
  %   \quad\text{for}\quad
  %   1\leq  i, j\leq  n.
  % \]
  %  Similarly, we denote $\eta K$  the square matrix $\diag(\eta) K$.
We now define the  effective spectrum and
effective spectral radius functions associated to $K$.

% We set $\Delta=[0,1]^n$ and  $\mathcal{E}(\Delta) = \{0,1\}^n$ the
%extreme points of $\Delta$.

\begin{Vdefi}
   \label{defi:effective}
   Let $K$ be  a square matrix of size  $n\in \N^*$.  The effective spectrum
   $\spec[K]$ 
   and effective spectral radius $R_e[K]$ functions  are defined
   on $\R_+^n$ by:
   \[
     \spec[K](\eta)=\sspec(K\cdot \Diag (\eta))
     \quad\text{and}\quad
     R_e[K](\eta)=\rho(K\cdot \Diag (\eta))
     \quad\text{for}\quad
          \eta\in\R_+^n.
   \]
 \end{Vdefi}
Trivially, two  matrices with  the  same
 effective spectrum have the same effective spectral radius.

 \medskip
    
 Motivated by the quantitative effect of vaccination strategies on the
 reproduction number in epidemic models, see
 \cite[Theorem~7.4]{EpidemicsInHeCairns1989} or 
 \cite{delmas_infinite-dimensional_2020}  in a more general
 framework, we give in \cite{ddz-Re}
 examples of transformations on positive compact operators that leave
 the effective spectral radius and effective spectrum functions
 invariant (mainly transposition and diagonal similarity, see below).
 The aim of this note is to explore further this invariance property in a
 finite dimensional setting, that is, for matrices.

\medskip
 
 Our main result is a characterization of the equality between
effective spectrum, shown in Section~\ref{sec:equality}.
 Building in
 particular on results from \cite{hartfiel,Loe86}, we then give
 in Section~\ref{sec:preservers} sufficient
 conditions for two matrices to have the same effective spectral
 radius, and show that they are necessary under various additional
 assumptions.

\section{Equivalent conditions for equality}
\label{sec:equality}

Let us first define some notation.
  For $\alpha$ and $\beta$ non-empty  subsets of $\{1,...,n\}$ we denote
  by  $K[\alpha,\beta]$ the  sub-matrix of  $K$ obtained  by keeping  the
  lines   in   $\alpha$   and   the  columns   in   $\beta$,   and   let
  $K[\alpha]  = K[\alpha,\alpha]$.
     The determinant  of $K[\alpha]$  is
  called a \emph{principal minor} of $K$,  its \emph{size} is the cardinal of
  $\alpha$. It is elementary to check that the characteristic polynomial
  of $K$ may be written as:
\begin{equation}
  \label{eq:charac_poly}
  \chi_K(t) = \sum_{k=0}^n (-1)^k  c_{n-k} t^k,
  \end{equation}
where   $c_0=1$ and, for
$j\geq 1$,  $c_j$ is the sum of all principal minors of size $j$ of $K$.

\medskip

We now give  equivalent conditions
for  the  effective  spectrum  and effective  spectral  radius  for  two
matrices to be equal.

\begin{theorem}[Effective spectrum and principal minors]
  \label{prop:minors}
  Let  $K$ and  $\tilde{K}$ be  square matrices  of the  same size $n\in
  \N^*$ with
  nonnegative entries.  The following are equivalent. 
    \begin{enumerate}[(i)]
    \item\label{lem:minors:R} The functions $R_e[K]$ and $
      R_e[\tilde{K}]$ coincide on $\R_+^n$. 
    \item\label{lem:minors:RE}  The functions $R_e[K]$ and $
      R_e[\tilde{K}]$ coincide on $\{0, 1\}^n$.
    \item\label{lem:minors:S} The functions $\spec[K]$ and $\spec[\tilde{K}]$ coincide on $\R_+^n$. 
    \item\label{lem:minors:SE}    The     functions    $\spec[K]$    and
      $\spec[\tilde{K}]$ coincide on  $\{0, 1\}^n$.
    \item\label{lem:minors:min} All principal minors of $K$ and $\tilde{K}$ coincide.
    \end{enumerate}
  \end{theorem}

For simplicity, we write $\ce(n)=\{0, 1\}^n$. 
  
    \begin{proof}%[Proof of Lemma \ref{lem:minors}]
      Clearly \ref{lem:minors:S}$\implies$\ref{lem:minors:R}$\implies$\ref{lem:minors:RE}, and \ref{lem:minors:S}$\implies$\ref{lem:minors:SE}$\implies$\ref{lem:minors:RE}.

      Let us check  that \ref{lem:minors:min} implies \ref{lem:minors:S}. Assume  that all principal
      minors of  $K$ and $\tilde{K}$  coincide.  
            For any vector $\eta\in \R^n$  and any set
      of indices $\alpha$, by multi-linearity of the determinant, we get:
      \[
        \det\Big((K\cdot \diag(\eta))[\alpha]\Big) = \Big(\prod_{i\in\alpha}
          \eta_i \Big) \det\left( K[\alpha]\right).  
      \] 
      Consequently,    all   principal    minors   of    $K\cdot \diag(\eta)$   and
      $\tilde{K}\cdot \diag(\eta)$   coincide.    By~\eqref{eq:charac_poly}   this
      implies that  $K\cdot \diag(\eta)$ and $\tilde{K}\cdot \diag(\eta)$ have  the same
      spectrum. Thus, Point
      \ref{lem:minors:S} holds.  \medskip
  
      Therefore, it is enough  to prove that \ref{lem:minors:RE} implies
      \ref{lem:minors:min}.  The proof is an induction on the dimension.
      The result  is clear in dimension  $1$.  Assume that it  holds for
      any square  matrix with  nonnegative entries of  dimension smaller
      than or equal to $n$.  Let  $K$ and $\tilde{K}$ be two square matrices of
      dimension  $n+1$ with nonnegative entries, and  assume that  $R_e[K]$ and  $R_e[\tilde{K}]$
      coincide    on   $\ce(n+1)$.     For   any    non-empty
      $\alpha\subset  \{1,...,n+1\}$, let  $\eta_\alpha$  be the  column
      vector   $(\un_\alpha(i),   1\leq   i\leq  n+1)$, that is, with 1
      in position $\alpha$ and 0 otherwise.    Notice that for
      any matrix $K'$:
\[
  R_e[K'](\eta_{\alpha}) = \rho(K' \cdot \diag(\eta_\alpha))=\rho(K'[\alpha]).
\]
Fix $\alpha\subset \{1,...,n+1\}$ nonempty, with $\alpha\neq
\{1,...,n+1\}$. Let $\beta\subset \alpha$ and set  $\tilde{\eta}_\beta =
(\un_\beta(i), i\in \alpha)$.  
We have:
 % \begin{equation}
 %   \label{eq:ReK}
 %    R_e[K'[\alpha]](  \tilde{\eta}_\beta)=  \rho(K'[\alpha]
 %    \cdot \diag(\tilde{\eta}_\beta))
 %    = \rho(K'\cdot \diag(\eta_\alpha)\cdot \diag(\eta_\beta)) =
 %    \rho(K'\cdot \diag(\eta_\beta))= 
 %    R_e[K'](\eta_\beta).
 %      \end{equation}     
 \begin{align}
   \label{eq:ReK}
   R_e[K'[\alpha]](  \tilde{\eta}_\beta)
   =  \rho(K'[\alpha]    \cdot \diag(\tilde{\eta}_\beta))
   = \rho(K'\cdot \diag(\eta_\alpha)\cdot \diag(\eta_\beta))
   &= \rho(K'\cdot \diag(\eta_\beta))
     \\   \nonumber
&=    R_e[K'](\eta_\beta).
      \end{align}
      Since   $\eta_\beta\in \ce(n+1)$, we  get
      $R_e[K](\eta_\beta)   =    R_e[\tilde{K}](\eta_\beta)$   for   all
      $\beta\subset  \alpha$.    We  deduce  from   \eqref{eq:ReK}  that
      $R_e[K[\alpha]]=R_e[\tilde K[\alpha]]$ on $\ce(\card \alpha)$.  By
      the induction  hypothesis the principal minors  of $K[\alpha]$ and
      $\tilde{K}[\alpha]$  are equal, that is all  principal minors of  size less
      than or equal to $n$ of $K$ and $\tilde{K}$ coincide. It remains to check that
      the  determinants are  the same.   Since all  principal minors of  size less
      than or equal to $n$ coincide, we deduce
      from~\eqref{eq:charac_poly} that:
  \begin{equation}
   \label{eq:caract-K=}
    \chi_K(t) - \det(K) = \chi_{\tilde{K}}(t) - \det(\tilde{K}).
  \end{equation}
  Since  $K$ and  $K'$  have nonnegative  entries, by  Perron-Frobenius
  theorem, their spectral  radius $\rho(K)=R_e[K](\ind{})$ and $\rho(K')=R_e[K'](\ind{})$
  is  also  an eigenvalue,  and  thus  a  root of  their  characteristic
  polynomial.     As    $R_e[K](\ind{})=R_e[K'](\ind{})$,   we    deduce
  from~\eqref{eq:caract-K=} that $\det(K) = \det(\tilde{K})$.  This ends
  the proof of the induction step.
\end{proof}

According to the proof of Theorem~\ref{prop:minors}, we have that~\ref{lem:minors:min} implies~\ref{lem:minors:S} and thus~\ref{lem:minors:R}, \ref{lem:minors:RE} and~\ref{lem:minors:SE}
without assuming that the entries are nonnegative.  We first
investigate whether \ref{lem:minors:R} from Theorem~\ref{prop:minors}
implies \ref{lem:minors:min} when the entries of the matrices have
generals signs.  Notice that $R_e[K]=R_e[\tilde{K}]$ automatically
implies $K$ and $\tilde K$ have the same entries on the diagonal up to
their sign (evaluate the effective spectral radii on $\eta$ with only
one non-zero component).  So in order for the equality of the
effective spectral radii of $K$ and $\tilde K$ to imply the equality
of all principal minors, it is necessary to assume that the two
matrices have the same sign on their diagonal, that is,
$\sgn(K_{ii})=\sgn(\tilde K_{ii})$ for all indices $i$. It is however
not enough, see next example and lemma.

  % We now  comment on the  nonnegativeness condition for the entries of the
  % matrices in Lemma~\ref{lem:minors}. 
  
\begin{Vex}[Same effective spectral radii do not imply same
  principal minors in general]
Consider the following two matrices:
 \[
   K=\begin{pmatrix}
      0&1\\1 & 0
    \end{pmatrix}
    \quad\text{and}\quad
      \tilde K=\begin{pmatrix}
      0&-1\\1&0
    \end{pmatrix}.
  \]
We have  $R_e[K]=R_e[\tilde K]$ on $\R_+^2$, but, even if
all  the principal  minors of 
size  1  coincide, the  principal minor of  size  two  is  different.
\end{Vex}
  
The key point is in fact the 
number of zeroes on the diagonal.

\begin{Vlem}
  \label{Vlem:Re-min}
 Let  $K$ and  $\tilde{K}$ be  square matrices  of the  same size $n\in
  \N^*$  with the  same sign on their diagonal
 and  having  at  most   one  zero  term in their diagonal. 
If  $R_e[K]=R_e[\tilde{K}]$ (on $\R_+^n$), then all principal minors of $K$ and $\tilde{K}$ coincide.
\end{Vlem}
  
\begin{proof}
Mimicking the  proof by
    induction of \ref{lem:minors:RE} $\implies$\ref{lem:minors:min} from
    Theorem~\ref{prop:minors} and  assuming~\ref{lem:minors:R}, we  get by
    induction that~\eqref{eq:caract-K=}  holds for $K$ and  $\tilde K$ as well
    as   for  $K\cdot \diag(\eta)$  and   $\tilde K\cdot \diag(\eta)$   with
    $\eta\in   \R_+^{n+1}$    (this   amounts   to   multiply    all   terms
    in~\eqref{eq:caract-K=} by $\prod _{i=1}^{n+1} \eta_i$):
  \[
 \chi_{K\eta} (t) - \det(K\cdot \diag(\eta)) =
 \chi_{\tilde{K}\eta}(t) - \det(\tilde{K}\cdot \diag(\eta)).
 \]
 As  there  is  at most  one  zero  on  the  diagonal, without  loss  of
 generality  (multiplying  $K$  and  $\tilde  K$ by  $-1$  and  using  a
 permutation of the canonical bassis of $\R^{n+1}$ if necessary), we can
 assume      that     $K_{11}=\tilde      K_{11}=a>     0$.       Taking
 $\eta=(1, \varepsilon, \ldots,  \varepsilon)$ for $\varepsilon>0$ small
 enough,  we deduce  that the  spectral radius  of $K\cdot  \diag(\eta)$
 (resp.  $\tilde{K}\cdot  \diag(\eta)$) is  also a simple  eigenvalue of
 $K\cdot   \diag(\eta)$  (resp.    $\tilde{K}\cdot  \diag(\eta)$).    As
 $R_e[K](\eta)=R_e[\tilde{K}](\eta)$,        we       deduce        that
 $\det(K\cdot \diag(\eta))=  \det(\tilde{K}\cdot \diag(\eta))$  and thus
 $\det(K) = \det(\tilde{K})$. Thus, by  induction, all the minors of $K$
 and $\tilde{K}$ coincide.
\end{proof}

We     now     check that      \ref{lem:minors:RE}     from
Theorem~\ref{prop:minors}         does not     imply            \ref{lem:minors:R}
or~\ref{lem:minors:min} when  the entries of the  matrices have generals
signs (even with positive entries on the diagonal).

\begin{Vex}[Same effective spectral radii on Boolean vectors do not imply same
  effective spectral radius]
  %nor same principal minors in general]
%  \label{rem:cex}

  % Property \ref{lem:minors:RE} from Lemma  \ref{lem:minors} does
  % not imply \ref{lem:minors:R} nor \ref{lem:minors:min}
Consider  the
  following two matrices:
 \[
   K=\begin{pmatrix}
      1&\beta\\\beta& 1 
    \end{pmatrix}
    \quad\text{and}\quad
      \tilde K=\begin{pmatrix}
      1&-\gamma\\\gamma&1
    \end{pmatrix},
  \]
  where $\gamma>0$  and $\beta= \sqrt{1+\gamma^2} -1$.   The eigenvalues
  of  $K$ are  $\sqrt{1+ \gamma^2}$  and $2  - \sqrt{1+  \gamma^2}$; the
  eigenvalues of  $\tilde K$ are $1  \pm \gamma i$.  In  particular, the
  two matrices have  the same spectral radius  $\sqrt{1+ \gamma^2}$. The
  functions $R_e[K]$ and $ R_e[\tilde{K}]$ clearly coincide on $\ce(2)$.
  Since $\det(K) \neq \det(\tilde K)$,  we deduce that all the principal
  minors  of   $K$  and   $\tilde  K$  do   not  coincide,   and  thus
  $R_e[K]\neq R_e[\tilde{K}]$ on $\R_+^2$ thanks to Lemma~\ref{Vlem:Re-min}.
  \end{Vex}

  \section{Matrices with the same effective spectral radius}
  \label{sec:preservers}

  Let us first recall a few  notions. 
  The  matrix   $K$   is
\emph{irreducible}  if  $K[\alpha,  \alpha^c]\neq  0$  for  all  subsets
$\alpha$ such  that $\alpha$ and  $\alpha^c$ are non-empty. 
The non-empty subset $\alpha$  is irreducible for $K$
if $K[\alpha]$ is irreducible.
Let $\ca(K)$ be the family  of 
maximal irreducible sets  for the inclusion, and consider the matrix
$K^\ca$  given by:
\[
  K^\ca_{i j}=K_{i j}
  \quad\text{if $i,j\in \alpha$ for some  $\alpha\in \ca(K)$,}\quad\text{and
    $K^\ca_{ij}=0$ otherwise.}
\]
The elements of $\ca(K)$ corresponds to  the %non-zero
atoms of $K$ in \cite{ddz-Re}. 
The map $K \mapsto K^\ca$ is not linear.

  The matrix
$K$ is  \emph{completely reducible}  if $K[\alpha,  \alpha^c]=0$ implies
$K[\alpha^c, \alpha]=0$ whenever $\alpha$  and $\alpha^c$ are non-empty,
or equivalently if $K=K^\ca$. 
We have  the following  graph interpretation: consider the oriented graph $G=(V,E)$ with $V=\{1, \ldots, n\}$
and $ij\in E$, that  is $ij$ is an oriented edge of $G$,  if and only if
$K_{ij}\neq 0$.  Then  the matrix $K$  is irreducible  if for any  choice of
vertices  $i,j\in V$  there is  an oriented  path from  $i$ to  $j$; the
matrix $K$ is completely reducible if  for any vertices $i,j\in V$ there
is an oriented path from $i$ to $j$  if and only if there is an oriented
path from $j$ to $i$.

\medskip

Recall  the matrix  $K$ is  \emph{diagonally similar}  to a
matrix $\tilde  K$ if there exists  a non singular real  diagonal matrix
$D$   such   that  $K=D \cdot  \tilde   K \cdot   D^{-1}$. Notice that if $K$ and
$\tilde K$ have nonnegative entries one can assume without loss of
generality that $D$ is also nonnegative. 
We recall the following well known result (see 
 \cite[Lemma~3.1 and Corrolary~5.4]{ddz-Re} in the infinite dimensional setting). For
 $\eta\in \R_+^n$, we denote $\ind{\{\eta>0\}}$ the vector whose $i$-th
 component is $\ind{\{\eta_i>0\}}$.

\begin{Vlem}[Sufficient conditions for equality of effective spectrum]
  \label{lem:Re}
  Let $K$ and  $\tilde K$ be  square matrices of  the same
  size $n\in \N^*$ with nonnegative  entries. We have:
  \begin{enumerate}[(i)]
   \item  $\spec[K]=\spec[K^\top]=\spec[K^\ca]$. 
%      \quad\text{and}\quad R_e[K]=R_e[ K^\top]=R_e[ K^\ca]$.
    \item If $K$ and $\tilde K$ are diagonally similar, then $\spec[K]=\spec[\tilde K]
      $.
      % and $R_e[K]=R_e[\tilde K]$.
    \item For $\eta\in \R_+^n$, we have:
 \begin{align*}
   \spec[K\cdot \diag(\eta)]
   =\spec[ \diag(\eta)\cdot K]
  & =\spec[\diag(\ind{\{\eta>0\}})\cdot K\cdot \diag(\eta)]
\\ &   =\spec[\diag(\eta)\cdot K\cdot \diag(\ind{\{\eta>0\}})].
      \end{align*}     
%    \quad\text{and}\quad  R_e[K\eta]=R_e[\eta
%    K]=R_e[\ind{\{\eta>0\}}K\eta]=R_e[\eta  K\ind{\{\eta>0\}}]$.  
  \end{enumerate}
\end{Vlem}

We now try to find necessary conditions for equality of effective spectra.
In other words, we would like to see if there are others transformations
of matrices that leave the effective spectrum invariant. 
   Following \cite{BouChe17}
   we introduce the notion of clan. 

\begin{Vdefi}[Clans and clan-free matrix]
  Let $K$ be a  square matrix of size $n$. A  subset $\alpha$ of $\{1,...,n\}$ is 
  a \emph{clan} if it satisfies $2\leq \mathrm{Card}(\alpha) \leq
  n-2$, and the submatrices $K[\alpha,\alpha^c]$ and $K[\alpha^c,\alpha]$ have rank at most $1$.
 The matrix $K$ is  \emph{clan-free} if there exists no clan. 
\end{Vdefi}
\begin{Vrem}
  \label{rem:clanfree}
  A square matrix of size $n\in \{1, 2, 3\}$ is automatically clan-free.
  \end{Vrem}

  The following proposition gathers known results on necessary conditions
  for equality of principal minors, and therefore of effective spectrum. 
\begin{Vprop}
  \label{prop:matrix_case}
  Let $K$ and $\tilde{K}$ be square matrices of the same size with
  nonnegative entries, and the same effective spectrum, that is, 
  $R_e[K] = R_e[\tilde{K}]$.
  \begin{enumerate}[(i)]
  \item\label{prop:mat-Ksym-one} If $K$ and $\tilde{K}$ are symmetric,
    then $\tilde{K}=K$. 
  \item\label{prop:mat-Ksym-two}
    If $K$ is symmetric, then $\tilde{K}^\ca$ is diagonally similar
    to $K$.
    \item\label{prop:mat-Kirr}  If $K$ is irreducible and clan-free, 
      then $\tilde{K}$ is diagonally similar to $K$ or to $K^\top$.
    \end{enumerate}
  \end{Vprop}
  \begin{proof}
  Thanks to   Theorem~\ref{prop:minors},   the
  principal minors  of $K$  and $\tilde{K}$  coincide. The  results then
  follow directly  from \cite[Theorem 3.5]{EngSch80}, for  the symmetric
  case, \cite[Theorem 3]{hartfiel} for the irreducible case when
  $n\leq 3$ (by Remark~\ref{rem:clanfree}, there can be no clan
  in this case), and \cite[Theorem 1]{Loe86} for the clan-free case
  when $n\geq 4$.
\end{proof}

Finally, as a corollary of Theorem~\ref{prop:minors}, we show that
the clan-free assumption is needed and get
an additional sufficient condition for equality. 
  Assume that $\alpha=\{1,...,m\}$ is a clan for $K$ (and thus $2\leq
  m\leq  n-2$). Then, there exists vectors $v$, $w$ of size $m$,
and $b$, $c$ of size $n-m$ such that $K$ may be written in block form as:
\begin{equation}
   \label{eq:K1}
  K = \begin{pmatrix}
    A & v  b^\top \\
    c w^\top & B
  \end{pmatrix}.
\end{equation}
The choice of $v,w, b,c$ is not unique in general. We  say that:
\begin{equation}
   \label{eq:K2}
 \tilde{K}  = \begin{pmatrix}
    A^\top & w  b^\top\\
    cv^\top & B
  \end{pmatrix}
\end{equation}
is a \emph{partial transpose} of $K$ (note that the partial transpose is
not unique in general).
\begin{Vrem} Such transformations have been considered in the special case
  where $v=w$ in \cite[Lemma 5]{Loe86}; see also \cite{BouChe17} where a similar
  transformation called \emph{clan reversal} is introduced for skew symmetric matrices. 
\end{Vrem}
 
\begin{Vprop}
  \label{prop:mat-Kclan}
  If $K$ is not clan-free, then we have $R_e[K] = R_e[\tilde{K}]$ for
  any partial transpose $\tilde{K}$ of $K$.
  \end{Vprop}

% The proof of this proposition, which is postponed to the end of this
% section,  hinges on the following characterization
% of matrices whose functions $R_e$ coincide. 

 % \medskip
 %  If two effective reproduction functions coincide on $\Delta$, the
 %  principal minors may not coincide in general if the entries of the
 %  matrices are signed (that is, Property \ref{lem:minors:R} from Lemma
 %  \ref{lem:minors} does not imply \ref{lem:minors:min}). 

 %  If    two    effective     reproduction    functions    coincide    on
 %  $\mathcal{E}(\Delta)$, they may not coincide on $\Delta$ nor does 
 %  the principal minors coincide in general if
 %  the   entries  of   the  matrices   are  signed   (that  is   Property
 %  \ref{lem:minors:RE}  from   Lemma  \ref{lem:minors}  does   not  imply
 %  \ref{lem:minors:R}  nor  \ref{lem:minors:min}). Indeed,  c

\begin{proof}%[Proof of Proposition~\ref{prop:matrix_case}]
  To prove Point \ref{prop:mat-Kclan}, suppose that  $K$ has a
  clan $\alpha$, and  let $\tilde{K}$ be a partial transpose  of $K$,
  so that $K$ and $\tilde{K}$ may be given by \eqref{eq:K1} and
  \eqref{eq:K2}.
For any $\lambda \notin \spec(B)$, using a  classical formula 
 for
determinants of block matrices, we get:
\begin{align*}
  \det(K-\lambda I)
  &= \det(A-\lambda I - vb^\top (B-\lambda I)^{-1} c w^\top) \det(B-\lambda I), \\
  \det(\tilde{K}-\lambda I )
  &= \det(A^\top-\lambda I - w b^\top (B-\lambda I)^{-1} cv^\top) \det(B-\lambda I) \\
 &= \det(A-\lambda I - vc^\top ((B-\lambda I)^{-1})^\top b w^\top) \det(B - \lambda I).
\end{align*}
Since $b^\top (B-\lambda I) ^{-1} c$ is a one-dimensional matrix, it is equal to its
transpose, so that $\det(K-\lambda I) = \det(\tilde{K} - \lambda I)$
are
equal for all $\lambda\notin \spec(B)$, and thus for all $\lambda\in \C$
by continuity.
Consequently,  the matrices  $K$ and $\tilde{K}$ have the  same
spectrum.
For
  any $\beta$, it is easily  seen that $K[\beta]$ and $\tilde{K}[\beta]$
  are  partial  transposes  of  each   other,  so  that  $K[\beta]$  and
  $\tilde{K}[\beta]$  also  have  the same  spectrum, and in
  particular
  the same spectral radius. Therefore $R_e[K]$ and $R_e[\tilde{K}]$
  coincide as \ref{lem:minors:R} and \ref{lem:minors:RE} are equivalent 
  in Theorem~\ref{prop:minors}.
\end{proof}

%\bibliographystyle{abbrv}
%\bibliography{biblio}

%\printbibliography

\end{document}